\documentclass[11pt,letterpaper]{article}
\usepackage{amsfonts, amsmath, amssymb, amscd, amsthm, color, graphicx, mathrsfs, wasysym, setspace, mdwlist, calc,float}
\usepackage{setspace}
\usepackage{hyperref}
\usepackage{hyperref}
\usepackage{xcolor}
\hypersetup{colorlinks,
    linkcolor={red!50!black},
    citecolor={blue!80!black},
    urlcolor={blue!80!black}}
\usepackage{float}

\newtheorem{thm}{Theorem}[section]
\newtheorem{cor}[thm]{Corollary}
\newtheorem{lem}[thm]{Lemma}
\newtheorem{prop}[thm]{Proposition}
\newtheorem{prob}[thm]{Problem}

\theoremstyle{definition}
\newtheorem{defn}[thm]{Definition}
\newtheorem{conv}[thm]{Convention}
\theoremstyle{remark}

\newtheorem{ex}[thm]{Example}

\renewcommand{\d }{{\rm d} }
\newcommand{\dl}{\widehat{\d}}
\newcommand{\G }{\Gamma (G, \mathcal A)}

\newcommand{\e }{\varepsilon }
\renewcommand{\kappa }{\varkappa}
\newcommand{\lab}{{\bf Lab}}

\newcommand{\h}{\hookrightarrow _h }
\newcommand{\sr}{{\bf sr}}

\newcommand{\C}{C_{r}^\ast}

\begin{document}

\title{On invertible elements in reduced $C^*$-algebras of acylindrically hyperbolic groups}
\author{M. Gerasimova\thanks{The first author was supported by ERC Consolidator Grant No. 681207}, D. Osin\thanks{The second author was supported by the NSF grant DMS-1612473.}}
\date{}
\maketitle

\begin{abstract}
Let $G$ be an acylindrically hyperbolic group. We prove that if $G$ has no non-trivial finite normal subgroups, then the set of invertible elements is dense in the reduced $C^\ast$-algebra of $G$. The same result is obtained for finite direct products of acylindrically hyperbolic groups.
\end{abstract}

\section{Introduction}

The topological stable rank of a $C^\ast$-algebra $A$, denoted $\sr(A)$, is a dimension like invariant introduced by Rieffel in \cite{Rie}. Recall that $\sr(A)$ is the minimal $n\in \mathbb N$ such that the set of all $n$-tuples of elements of $A$ that generate $A$ as a left ideal is dense in $A^n$; if no such $n$ exists, then $\sr(A)=\infty $. In particular, $\sr(A)=1$ if and only if the group of invertible elements $GL(A)$ is dense in $A$.

The study of stable rank of $C^\ast$-algebras is partially motivated by applications to the $K$-theory. For instance, if $A$ is a unital $C^\ast$-algebra and $\sr(A)=1$, then any two projections representing the same elements in $K_0(A)$ are homotopic \cite{B} and $K_1(A)=GL(A)/GL^0(A)$, where $GL^0(A)$ is the connected component of $GL(A)$ containing $1$ \cite{Rie87}.

All groups considered in this paper are countable by default. Given such a group $G$, we denote by $C^\ast_r(G)$ its reduced $C^\ast$-algebra. Answering a question posed in \cite{Rie}, Dykema, Haagerup, and R\o rdam \cite{DHH} proved that $C^\ast_r(F_n)$ has stable rank one, where $F_n$ is a free group of rank $1\le n\le \infty$. This result was later generalized to torsion free hyperbolic groups and certain free products with finite amalgamated subgroups by Dykema and de la Harpe \cite{DH}. The main goal of this paper is to show that the same property holds for a much larger class of groups acting on hyperbolic spaces.

Recall that an isometric action of a group $G$ on a metric space $S$ is {\it acylindrical} if for every $\e>0$, there exist $R,N>0$
such that for every two points $x,y\in S$ with $\d (x,y)\ge R$, there are at most $N$ elements $g\in G$ satisfying
$$
\d(x,gx)\le \e \;\;\; {\rm and}\;\;\; \d(y,gy) \le \e.
$$
An isometric action of a group $G$ on a hyperbolic space $S$ is \emph{non-elementary}, if the limit set  $\Lambda (G)\subseteq \partial S$ is infinite. For acylindrical actions, being non-elementary is equivalent to the action having unbounded orbits and $G$ being not virtually cyclic \cite[Theorem 1.1]{Osi16}. A group $G$ is \emph{acylindrically hyperbolic} if it admits a non-elementary acylindrical action on a hyperbolic space.

Examples of acylindrically hyperbolic groups include non-elementary hyperbolic and relatively hyperbolic groups, mapping class groups of closed surfaces of non-zero genus, $Out(F_n)$ for $n\ge 2$, non-virtually cyclic groups acting properly on proper $CAT(0)$ spaces and containing a rank-$1$ element, groups of deficiency at least $2$, most $3$-manifold groups, automorphism groups of some algebras (e.g., the Cremona group of birational transformations of the complex projective plane) and many other examples. For a more detailed discussion we refer to the survey \cite{Osi18}.

Every acylindrically hyperbolic group $G$ contains a unique maximal finite normal subgroup $K(G)$ called the \emph{finite radical} of $G$ \cite[Theorem 6.14]{DGO}. Our main result is the following.

\begin{thm}\label{main}
Let $G_1, \ldots, G_k$ be acylindrically hyperbolic groups with $K(G_i)=\{1\}$ for all $1\le i\le k$. Then $\sr(\C(G_1\times \cdots \times G_k))=1$. In particular, the reduced $C^\ast$-algebra of any acylindrically hyperbolic group with trivial finite radical has stable rank $1$.
\end{thm}

This result is new even for $k=1$ and covers previously known result from \cite{DH}. It also shows that the behaviour of $\sr(\C(G_1\times \cdots \times G_k))$ for acylindrically hyperbolic groups is in sharp contrast to the case when $G_1, \ldots, G_k$ are abelian. Indeed, the Gelfand-Neimark theorem and basic facts from dimension theory imply that  $$\sr(\C(\mathbb Z^k))=[k/2]+1\to \infty $$ as $k\to \infty$ \cite{Rie}.

We note that the reduced $C^\ast$-algebras of products of acylindrically hyperbolic groups with trivial finite radical are always simple. Indeed, this is an easy consequence of \cite[Theorem 2.35]{DGO} and \cite[Theorem 1.4]{BKKO}. In general, every $r\in \mathbb N\cup \{\infty\}$ realizes as the stable rank of a simple $C^\ast$-algebra \cite{V}. However, we are not aware of any example of a group $G$ such that $\C(G)$ is simple and $\sr(\C(G))>1$.

The proof of Theorem \ref{main} follows the general strategy suggested in \cite{DHH} and \cite{DH}. The crucial ingredient used in these papers is the \emph{$\ell^2$-spectral radius property} for elements of $\C(G)$, which is derived from Jolissant's property of rapid decay in case $G$ is hyperbolic or from the tree-like structure in case $G$ is an amalgamated free product. Unfortunately, neither of these two approaches works for general acylindrically hyperbolic groups.

To overcome this problem, we suggest a geometric method of bounding the operator norm of elements of $\mathbb CG$ from above  (see Section 2) inspired by the work of Catterji--Ruane and Sapir on property (RD) \cite{CR, S}. This approach requires constructing \emph{generalized combings}, i.e. maps from $G\times G$ to the set of all subsets of $G$, with certain additional properties. We show that every acylindrically hyperbolic group admits such a generalized combing in Section 4. This is the most technical part of our work, which makes use of the notion of a \emph{hyperbolically embedded subgroup} introduced in \cite{DGO}. To make our paper as self-contained as possible, we review the necessary background in Section 3. Finally, we combine the results obtained in Sections 2 and 4 to prove our main theorem in Section 5.

We conclude with the following question.
\begin{prob}
Suppose that a group $G$ splits as
$$
1\longrightarrow K \longrightarrow G\longrightarrow Q\longrightarrow 1,
$$
where $K$ is finite and $\sr(\C(Q))=1$. Does it follow that $\sr(\C(G))=1$?
\end{prob}
The affirmative answer to this question together with Theorem \ref{main} would imply that $\sr(\C(G))=1$ for any acylindrically hyperbolic group $G$, as well as for any direct product of such groups. It is worth noting that in the simplest case $G=K\times Q$, the equality $\sr(\C(G))=1$ follows from results of \cite{Rie}.

\section{Bounding the operator norm via generalized combings}

The main goal of this section is to develop geometric tools for bounding the operator norm of elements of a group algebra in terms of their $\ell^2$-norm. We begin by recalling necessary definitions and notation.

Given a countable group $G$, let $\ell^2(G)$ denote the set of all square-summable functions $f\colon G\to \mathbb C$. By $\lambda_G\colon \mathbb CG\to B(\ell^2(G))$ we denote the left regular representation of $\mathbb CG$, where $B(\ell^2(G))$ is the set of all bounded operators on $\ell^2(G)$.  For an element $a\in \mathbb CG$, we denote by $\| a\|_2$ its $\ell^2$-norm and by $\| a\|$ the operator norm of $\lambda_G(a)\in B(\ell^2(G))$. That is
$$
\| a\| =\sup\limits_{v\in \ell^2(G)\setminus \{ 0\}} \frac{\| av\|_2}{\| v\|_2}.
$$

The \emph{reduced $C^\ast $-algebra of a group $G$}, denoted $C^\ast_{red}(G)$, is the closure of $\lambda_G (\mathbb CG)$ in $ B(\ell^2(G))$ with respect to the operator norm. The involution on $C^\ast_{red}(G)$ is induced by the standard involution on $\mathbb CG$:
$$\left(\sum\limits_{g\in G} \alpha_g g\right)^\ast = \sum\limits_{g\in G} \bar\alpha_g g^{-1},$$ where $\alpha _g\in \mathbb C$ for all $g\in G$.

Finally, we let
$r(a)$ and $r_2(a)$ denote the spectral radii of an element $a\in \C(G)$ corresponding to the operator norm and the $\ell^2$-norm, respectively. That is,
$$
r(a)=\lim_{k\to \infty}\sqrt[k]{\| a^k\|}
$$
and
$$
r_2(a)=\limsup_{k\to \infty}\sqrt[k]{\| a^k\|_2}.
$$
As $\|a\|_2 \leq \|a\|$, we clearly have
\begin{equation}\label{r2r}
r_2(a) \leq r(a).
\end{equation}

Recall that a \emph{combing} on a group $G$ generated by a set $A$ is a map that to each pair of points $g,h\in G$ assigns a path $\gamma_{g,h}$ in the Cayley graph of $G$ with respect to $A$ connecting $g$ to $h$. Combings on groups were introduced in the pioneering work of Epstein and Thurston \cite{ET} and played fundamental role is several branches of geometric group theory, including the theory of automatic and semihyperbolic groups. Observe that every combing on $G$ yields a map $G\times G \to \mathcal{P}(G)$, where $\mathcal{P}(G)$ is the set of all subsets of $G$, via the identification of the path $\gamma_{g,h}$ with its set of vertices. This interpretation leads to the following generalization.

\begin{defn}
Let $G$ be a group. A \emph{generalized combing} of $G$ is a map $C\colon G\times G \to \mathcal{P}(G)$. We say that the combing $C$ is
\begin{enumerate}
\item[---] \emph{symmetric} if $C(x,y)=C(y,x)$ for all $x,y\in G$;
\item[---] \emph{$G$-equivariant} if $C(gx,gy)=gC(x,y)$ for all $x,y,g\in G$.
\end{enumerate}
\end{defn}

In this section, we will also use a generalization of length functions of groups.

\begin{defn}
A map $\ell\colon G\to [0, +\infty)$ is a \emph{pseudolength function} on a group $G$ if it is symmetric and satisfies the triangle inequality; that is, $\ell(g^{-1})=\ell(g)$ and $\ell(gh)\le \ell(g) + \ell(h)$ for all $g,h\in G$. We say that $\ell$ is a length function if, in addition, $\ell (g)=0$ if and only if $g=1$.
\end{defn}
Pseudolength functions naturally occur from group actions on metric spaces. Indeed, if a group $G$ acts isometrically on a metric space $(X,\d)$, then fixing a basepoint $x\in X$ we get a pseudolength function $\ell(g)= \d (x, gx)$. Another natural class of examples consists of word lengths on $G$ with respect to fixed generating sets. In fact, considering word lengths would suffice for the purpose of proving Theorem \ref{main}. We choose to work with pseudolength functions because the proof is essentially the same and our results can be potentially applied in the more general context of groups acting on metric spaces.

From now on, we fix a group $G$ and a pseudolength function $\ell\colon G\to [0,+\infty)$. For every $n\in \mathbb N$, we define
$$
B (n)=\{ g\in G\mid \ell(g) \le n\}.
$$
To each generalized combing $C\colon G\times G \to \mathcal{P}(G)$, we associate two growth functions $\gamma, \rho\colon \mathbb N\to \mathbb N\cup\{\infty\}$ defined as follows.
Let
\begin{equation}\label{gamma}
\gamma (n) = \sup\limits_{g\in G} |C(1,g)\cap B (n)|
\end{equation}
and
\begin{equation}\label{delta}
\rho (n)= \sup\limits_{g\in B(n)} \sup_{x\in  C(1,g)} \ell(x).
\end{equation}
Note that, in general, these functions can take infinite values.

We will need the following elementary observation.

\begin{lem}\label{C1g}
Let $C\colon G\times G\to \mathcal P(G)$ be a generalized combing such that the functions $\gamma$ and $\rho$ take only finite values. Then for any element $s\in B(n)$, we have $|C(1,s)|\le  \gamma(\rho(n))$
\end{lem}
\begin{proof}
Note that (\ref{gamma}) and (\ref{delta}) imply that $C(1,s)\subseteq B(\rho(n))$ for every $s\in B(n)$. Hence,
$|C(1,s)|= |C(1,s)\cap B(\rho(n))|\le \gamma(\rho(n))$.
\end{proof}

We are now ready to state the main result of this section. In the particular case $S=G$, it is similar to \cite[Proposition 1.7]{CR}. The proof is inspired by the approach suggested in \cite{S}.

\begin{prop}\label{CR}
Let $G$ be a group endowed with a pseudolength function $\ell\colon G\to [0, +\infty)$ and let $S$ be a subset of $G$. Suppose that there exists a symmetric $G$-equivariant generalized combing ${C\colon G\times G \to \mathcal{P}(G)}$ such that
\begin{equation}\label{Cint}
C(1,s) \cap C(s,g) \cap C(1,g) \neq \emptyset
\end{equation}
for all $s\in S$ and $g\in G$ and the associated growth functions $\gamma$ and $\rho$ take only finite values. Then for every $a\in \mathbb CG$ and $n\in \mathbb N$ such that $supp (a)\subseteq S\cap B(n)$, we have
\begin{equation}\label{||a||}
\|a\| \le \gamma (\rho(n))^{3/2} \| a\| _2.
\end{equation}
In particular, if $S$ is a subsemigroup of $G$ and $\lim\limits_{k\to \infty} \sqrt[k]{\gamma(\rho(k))}=1$, then $r(a)=r_2(a)$.
\end{prop}

Let $\mathbb R_+ G$ denote the subset of the group algebra $\mathbb CG$ consisting of linear combinations of elements of $G$ with positive real coefficients. The main step in the proof of Proposition \ref{CR} is the following.

\begin{lem}\label{r+}
Under the assumptions of Proposition \ref{CR}, suppose additionally that $a\in \mathbb R_+G$. Then for every $b\in \mathbb R_+ G$, we have
\begin{equation} \label{ab}
\|ab\|_2 \le \gamma(\rho(n))^{3/2} \|a\|_2\| b\|_2.
\end{equation}
\end{lem}

\begin{proof}
We fix arbitrary $n\in \mathbb N$ and $a\in \mathbb CG$ such that $supp (a)\subseteq S\cap B(n)$.
To simplify our notation, we introduce the following sets. For every $g\in G$, let
$$X_{g}=B(\rho(n))\cap C(1,g). $$
Further, for every $g\in G$ and $x\in X_g$, let
$$S_{g,x}=\{ s\in supp(a) \mid x\in C(1,s)\cap C(s,g)\}.$$

Note that for any $g\in G$, we have
\begin{equation}\label{suppa}
supp (a)\subseteq \bigcup\limits_{x\in X_g} S_{g,x}.
\end{equation}
Indeed, for every $s\in supp(a)\subseteq B(n)$,  we have $C(1,s)\subseteq B(\rho(n))$ by the definition of $\rho(n)$. By (\ref{Cint}), the intersection
$C(1,s) \cap C(s,g) \cap C(1,g)\cap B(\rho(n))=X_g \cap C(1,s) \cap C(s,g)$
is non-empty; taking any element $x$ from this intersection, we obtain $s\in S_{g,x}$ and (\ref{suppa}) follows.

Let
$$
a=\sum\limits_{g\in G} \alpha_gg, \;\;\; {\rm and}\;\;\; b=\sum\limits_{g\in G} \beta_gg.
$$
Observe that $\left| X_g\right| \leq \gamma(\rho(n))$ for any $g\in G$ by the definition of $\gamma(n)$. Using (\ref{suppa}) and then the inequality between the arithmetic and quadratic means, we obtain

\begin{align*}
\|ab\|_2^2=& \sum\limits_{g\in G}\left(\sum\limits_{s\in supp(a)}\alpha_s\beta_{s^{-1}g}\right)^2 \leq \\
&  \sum\limits_{g\in G}\left(\sum\limits_{x\in X_g}\sum\limits_{s\in S_{g,x}} \alpha_s\beta_{s^{-1}g}\right)^2  \leq\\
& \gamma(\rho(n))\sum\limits_{g\in G} \sum \limits_{x\in X_g} \left(\sum \limits_{s\in S_{g,x}} \alpha_s\beta_{s^{-1}g} \right)^2.
\end{align*}

Given $g, x\in G$, let
$$
T_{g,x}=\{ s^{-1}g\mid s\in S_{g,x}\}.
$$
Using the Cauchy-Schwarz inequality and substituting $t=s^{-1}g$, we obtain

\begin{align*}
\|ab\|_2^2 \; \leq\;  & \gamma(\rho(n))\sum_{g \in G} \sum\limits_{x\in X_g}\left( \sum\limits_{s\in S_{g,x}} \alpha_s^2\right)\left( \sum\limits_{s \in S_{g,x}} \beta_{s^{-1}g}^2\right) = \\
& \gamma(\rho(n))\sum_{g \in G} \sum\limits_{x\in X_g}\left( \sum\limits_{s\in S_{g,x}} \alpha_s^2\right)\left( \sum\limits_{t \in T_{g,x}} \beta_{t}^2\right)= \\
& \gamma(\rho(n)) \sum\limits_{s\in G} \sum\limits_{t\in G} C_{s,t} \alpha_s^2\beta_{t}^2,
\end{align*}
for some $C_{s,t}\ge 0$.

We now estimate the coefficients $C_{s,t}$. To this end, we note that every individual term $\alpha_s^2\beta_t^2$ occurs in the product $\left( \sum\limits_{s\in S_{g,x}} \alpha_s^2\right)\left( \sum\limits_{t \in T_{g,x}} \beta_{t}^2\right)$ at most once. Therefore, for every fixed $s$ and $t$, $C_{s,t}$ is bounded by the number of pairs $(g,x)\in G\times G$ satisfying the conditions $x\in X_g$,
\begin{equation}\label{s}
s\in S_{g,x},
\end{equation}
and
\begin{equation}\label{t}
t\in T_{g,x}.
\end{equation}
By the definition of $S_{g,x}$, (\ref{s}) implies that $s\in supp (a)$ and $x\in C(1,s)$. Since $supp(a)\subseteq B(n)$, we obtain $|C(1,s)| \le \gamma(\rho(n))$ by Lemma \ref{C1g}. Thus, there are at most $\gamma(\rho(n))$ elements $x$ satisfying (\ref{s}) for every fixed $s$.

Further, we fix $x$ and $t$ and note that (\ref{t}) is equivalent to $gt^{-1} \in S_{g,x}$. In turn, this is equivalent to $gt^{-1} \in supp(a)$ and $x\in C(1, gt^{-1})\cap C(gt^{-1},g)$. Since $\ell$ is symmetric, the former condition implies $\ell(tg^{-1})= \ell(gt^{-1})\le n$; hence $C(1, tg^{-1})\subseteq B(\rho(n))$. Now the latter condition and (\ref{delta}) yield
\begin{align*}
g^{-1}x \in & C(g^{-1}, t^{-1}) \cap C(t^{-1}, 1) = t^{-1}( C(tg^{-1}, 1)\cap C(1,t))= \\ & t^{-1}( C(1, tg^{-1})\cap C(1, t) \subseteq t^{-1} (B(\rho(n)) \cap C(1, t)).
\end{align*}
(here we use our assumption that $C$ is symmetric and $G$-equivariant). By the definition of $\gamma$, we have $|B(\rho(n)) \cap C(1, t)|\le \gamma(\rho(n))$. It follows that there are at most $\gamma(\rho(n))$ elements $g$ satisfying (\ref{t}) for any fixed $t$ and $x$. Thus, we have $C_{s,t}\le \gamma(\rho(n))^2$.

Finally, we obtain
$$
\| ab\|_2^2\le \gamma(\rho(n)) \sum\limits_{s\in G} \sum\limits_{t\in G} C_{s,t} \alpha_s^2\beta_{t}^2 \le \gamma(\rho(n))^3 \sum\limits_{s\in G} \sum\limits_{t\in G} \alpha_s^2\beta_{t}^2 = \gamma(\rho(n))^3 \|a\|_2^2 \| b\|_2^2.
$$
\end{proof}

We are now ready to prove our main result.

\begin{proof}[Proof of Proposition \ref{CR}]
Given an element $f=\sum\limits_{g\in G} \phi_g g \in \mathbb CG$, we define $f^+  =\sum\limits_{g\in G} |\phi _g| g$. Since $\mathbb{C}G$ is dense in $\ell^2(G)$ and $\|ab\|_2 = \|(ab)^+\|_2\leq \| a^+b^+\|_2$ for every $b\in \mathbb CG$, we have

$$
\| a\| = \sup\limits_{b\in \mathbb CG \setminus \{ 0\}} \frac{\|ab\|_2}{\|b\|_2} \le \sup\limits_{b\in \mathbb CG \setminus \{ 0\}} \frac{\|a^+b^+\|_2}{\|b^+\|_2} = \sup_{c\in \mathbb R_+G\setminus \{ 0\}} \frac{\|a^+ c\|_2}{\| c\|_2}.
$$
Applying Lemma \ref{r+} to the elements $a^+$ and $c$  we obtain that
$$\|a\|\leq \gamma(\rho(n))^{3/2} \|a^+\|_2 = \gamma(\rho(n))^{3/2} \|a\|_2.$$

The claim about the spectral radii is an immediate consequence of (\ref{||a||}) and the definitions of $r(a)$ and $r_2(a)$. Indeed, assume that $S$ is a subsemigroup of $G$. Then $supp(a^k)\subseteq S\cap B(nk)$ for all $k\in \mathbb N$ since $\ell$ satisfies the triangle inequality. Applying  (\ref{||a||}) to the element $a^k$, we obtain
$$
r(a)=\lim_{k\to \infty}\sqrt[k]{\| a^k\|}\le \limsup_{k\to \infty}\sqrt[k]{\gamma(\rho(nk))\| a^k\|_2} = \limsup_{k\to \infty}\sqrt[k]{\| a^k\|_2}=r_2(a).
$$
Combining this with (\ref{r2r}) we obtain $r_2(a)=r(a)$.
\end{proof}

\section{Preliminaries on hyperbolically embedded subgroups}

We begin by recalling necessary definitions and results on acylindrically hyperbolic groups and their hyperbolically embedded subgroups. Our main references are \cite{DGO} and \cite{Osi16}. Terminology and technical tools discussed here go back to \cite{Osi06} and \cite{Osi07}, where they were developed in the particular case of relatively hyperbolic groups.

Let $\mathcal A$ be a set, which we refer to as an \emph{alphabet}, given together with a (not necessarily injective) map $\alpha\colon \mathcal A\to G$ to a group $G$. We say that $\mathcal A$ is a \emph{generating alphabet} of $G$ if $G$ is generated by $\alpha(\mathcal A)$. Note that a generating set $X\subseteq G$ can be thought of as a generating alphabet with the obvious injection $X\to G$.

The \emph{Cayley graph} of $G$ with respect to a generating alphabet $\mathcal A$, denoted $\Gamma (G, \mathcal A)$, is a graph with vertex set $G$ and the set of edges defined as follows. For every $a\in \mathcal A$ and every $g\in G$, there is an oriented edge $(g, g\alpha(a))$ in $\Gamma (G, \mathcal A)$ labelled by $a$. Given a combinatorial path $p$ in $\Gamma (G, \mathcal A)$, we denote by $\lab (p)$ its label. Note that if $\alpha $ is not injective, $\Gamma (G, \mathcal A)$ may have multiple edges.

Suppose now that we have a group $G$, a subgroup $H$ of $G$, and a relative generating set $X$ of $G$ with respect to $H$; that is, we assume that $X$ and $H$ together generate $G$. We think of $X$ and $H$ as abstract sets and consider the disjoint union
\begin{equation}\label{calA}
\mathcal A= X \sqcup H,
\end{equation}
and the map $\alpha \colon \mathcal A\to G$  induced by the inclusions $X\to G$ and $H\to G$. By abuse of notation, we do not distinguish between subsets $X$ and $H$ of $G$ and their preimages in $\mathcal A$. This will not create any problems since the restrictions of $\alpha$ on $X$ and $H$ are injective. Note, however, that $\alpha $ is not necessarily injective. Indeed if $X$ and $H$ intersect in $G$, then every element of $H \cap X\subseteq G$ will have at two preimages in $\mathcal A$: one in $X$ and another in $H$ (since the union in (\ref{calA}) is disjoint).

\begin{conv}
Henceforth we always assume that generating sets and relative generating sets are symmetric. That is, if $x\in X$, then $x^{-1}\in X$. In particular, every element of $G$ can be represented by a word in $\mathcal A$.
\end{conv}

In these settings, we consider the Cayley graphs $\G $ and $\Gamma (H, H)$ and naturally think of the latter as a subgraph of the former. We introduce a (generalized) metric $$\dl \colon H \times H \to [0, +\infty]$$ by letting $\dl (h,k)$ be the length of a shortest path in $\G $ that connects $h$ to $k$ and contains no edges of $\Gamma (H, H)$. If no such a path exists, we set $\dl (h,k)=\infty $. Clearly $\dl $ satisfies the triangle inequality, where addition is extended to $[0, +\infty]$ in the natural way.

\begin{defn}\label{hedefn}
A subgroup $H$ of $G$ is \emph{hyperbolically embedded  in $G$ with respect to a subset $X\subseteq G$}, denoted $H \h (G,X)$, if the following conditions hold.
\begin{enumerate}
\item[(a)] The group $G$ is generated by $X$ together with $H$ and the Cayley graph $\G $ is hyperbolic, where $\mathcal A=X\sqcup H$.
\item[(b)] Any ball (of finite radius) in $H$ with respect to the metric $\dl$ contains finitely many elements.
\end{enumerate}
Further we say  that $H$ is \emph{hyperbolically embedded} in $G$ and write $H\h G$ if $H\h (G,X)$ for some $X\subseteq G$.
\end{defn}

\begin{figure}
 \begin{center}
  \scalebox{0.9}{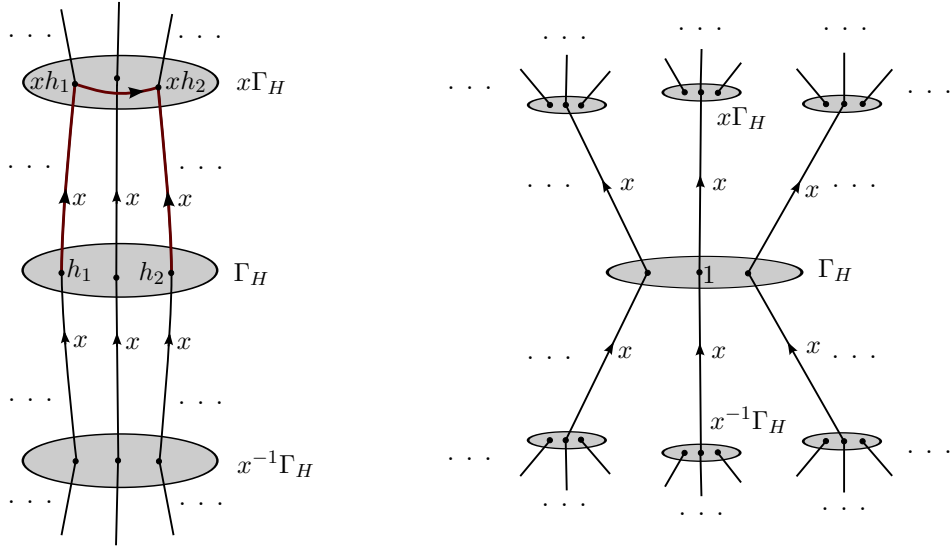}
  \caption{Cayley graphs $\Gamma(G, \mathcal A )$ for $G=H\times \mathbb Z$ and $G=H\ast \mathbb Z$.}\label{fig00}
 \end{center}
\end{figure}

Note that for any group $G$ we have $G\h G$.  Indeed we can take $X=\emptyset $ in this case. Further, if $H$ is a finite subgroup of a group $G$, then $H\h G$. Indeed $H\h (G,X)$ for $X=G$. These two cases are usually referred to as \emph{degenerate}. Since the notion of a hyperbolically embedded subgroup and the metrics $\dl$ play important roles in our paper, we consider two additional examples borrowed from \cite{DGO}.

\begin{ex}\label{bex}
\begin{enumerate}
\item[(a)]
Let $G=H\times \mathbb Z$, $X=\{ x\} $, where $x$ is a generator of $\mathbb Z$. Then $\Gamma (G, \mathcal A )$ is quasi-isometric to a line and hence it is hyperbolic. However, every two elements $h_1, h_2\in H$ can be connected by a path of length at most $3$ in $\G$ that avoids edges of $\Gamma _H=\Gamma (H,H)$ (see Fig. \ref{fig00}). Thus $H\not\h (G,X)$ whenever $H$ is infinite.

\item[(b)]  Let $G=H\ast \mathbb Z$, $X=\{ x\} $, where $x$ is a generator of $\mathbb Z$. In this case $\Gamma (G, \mathcal A )$ is quasi-isometric to a tree and no path connecting $h_1, h_2\in H$ and avoiding edges of $\Gamma_H=\Gamma (H,H)$ exists unless $h_1=h_2$ (see Fig. \ref{fig00}). Thus $H\h (G,X)$.
\end{enumerate}
\end{ex}

The idea behind the first example can also be used to prove the following more general result.

\begin{prop}[{\cite[Proposition 2.10]{DGO}}]\label{malnorm}
Let $G$ be a group and let $H\h G$. Then for every $g\in G\setminus H$, we have $|g^{-1}Hg\cap H|<\infty$.
\end{prop}

Hyperbolically embedded subgroups were introduced and studied in \cite{DGO} as a generalizaton of relatively hyperbolic groups; indeed, $G$ is hyperbolic relative to $H$ if and only if $H\h(G, X)$ for some $|X|<\infty$ \cite[Proposition 4.28]{DGO}. Other non-trivial examples occur in acylindrically hyperbolic groups. In fact, a group $G$ is acylindrically hyperbolic if and only if it contains a non-degenerate hyperbolically embedded subgroup \cite{Osi16}.

In this paper, we will use some special hyperbolically embedded subgroups constructed in \cite{DGO}. Recall that every acylindrically hyperbolic group contains a maximal finite normal subgroup $K(G)\le G$, called the \emph{finite radical} of $G$. This fact, as well, as the result below, are proved in \cite[Theorem 2.24]{DGO}.

\begin{thm}\label{Eg}
Let $G$ be an acylindrically hyperbolic group with trivial finite radical. Then for every $n\in \mathbb N$, there exists a free subgroup $F\le G$ of rank $n$ such that $F\h G$.
\end{thm}

The main technical tool used in our paper is Proposition \ref{n-gon} below. To state it we need two auxiliary definitions. Throughout the rest of this section we assume that $H\h(G,X)$ and use the notation introduced above.

Let $q$ be a path in the Cayley graph $\G $. A (non-trivial) subpath $p$ of $q$ is called an \emph{$H$-subpath}, if the label of $p$ is a word in the alphabet $H$. An $H$-subpath $p$ of $q$ is an {\it $H$-component} if $p$ is not contained in a longer $H$-subpath of $q$; if $q$ is a loop, we require, in addition, that $p$ is not contained in any longer $H$-subpath of a cyclic shift of $q$.

Two $H$-components $p_1, p_2$ of a path $q$ in $\G $ are called {\it connected} if there exists a path $c$ in $\G $ that connects some vertex of $p_1$ to some vertex of $p_2$, and ${\lab (c)}$ is a word consisting entirely of letters from $H$. In algebraic terms this means that all vertices of $p_1$ and $p_2$ belong to the same left coset of $H$. Note also that we can always assume that $c$ is an edge as every element of $H$ is included in the set of generators. An $H$-component of a path $p$ is called \emph{isolated} in $p$ if it is not connected to any other $H$-component of $p$.

It is convenient to extend the relative metric $\dl $ defined above to the whole group $G$ by assuming
\begin{equation}\label{defdl}
\dl (f,g)\colon =\left\{\begin{array}{ll}\dl (f^{-1}g,1),& {\rm if}\;  f^{-1}g\in H \\ \dl (f,g)=\infty , &{\rm  otherwise.}\end{array}\right.
\end{equation}

The following result is a simplified version of \cite[Proposition 4.13]{DGO}. By a \emph{geodesic $n$-gon} in $\G$ we mean a loop which is a concatenation of $n$ geodesics; these geodesics are referred to as \emph{sides} of $\mathcal P$. Given a combinatorial path $p$ in $\G$, we denote by $p_-$ and $p_+$ its beginning and ending vertices, respectively.

\begin{prop}\label{n-gon}
Let $G$ be a group, $H$ a subgroup of $G$. Suppose that $H\h (G,X)$ for some $X\subseteq G$ and let $\mathcal A=X\sqcup H$. Then there exists a constant $C$ satisfying the following conditions.   For any geodesic $n$-gon $p$ in $\G$ with sides $p_1, \ldots , p_n$ and any $I\subseteq \{ 1, \ldots, n\}$ such that $p_i$ is an isolated $H$-component of $p$ for all $i\in I$, we have $$\sum_{i\in I} \dl ((p_i)_-, (p_i)_+)\le Cn.$$
\end{prop}

\section{Generalized combings in acylindrically hyperbolic groups}

Recall that every acylindrically hyperbolic group contains a unique maximal finite normal subgroup $K(G)$ called the \emph{finite radical} of $G$ \cite[Theorem 6.14]{DGO}. The main goal of this section is to prove the following.

\begin{prop}\label{combing}
Let $G$ be an acylindrically hyperbolic group with trivial finite radical, $F$ a non-empty finite subset of $G$. Then there exists a length function $\ell$ on $G$, an element $t\in G$, and a symmetric $G$-equivariant generalized combing $C\colon G\times G\to \mathcal P(G)$ satisfying the following conditions.
\begin{enumerate}
\item[(a)] The set $tF$ freely generates a free subsemigroup $S$ of $G$.
\item[(b)] $C(1,s) \cap C(s,g) \cap C(1,g) \neq \emptyset$ for all $g\in G$ and $s\in S$.
\item[(c)] The associated growth functions $\gamma$ and $\rho$ computed with respect to the length function $\ell$ (see (\ref{gamma}) and (\ref{delta})) are bounded by a linear function from above.
\end{enumerate}
\end{prop}

In fact, the proposition is also true for $F=\emptyset$. However, the proof in this case is formally different and we exclude the possibility $F=\emptyset$ for the sake of brevity.

Note that it suffices to prove the proposition for a shift $gF$ for some $g\in G$. Since $F$ is finite, there is an element $g\in G\setminus F^{-1}$. Clearly $1\notin gF$ for such $g$. Thus we can assume that
\begin{equation}\label{1F}
1\notin F
\end{equation}
without loss of generality.

The proof will be divided into a sequence of lemmas, all of which are proved under the assumptions of Proposition \ref{combing}. Given a combinatorial path $p$ in a Cayley graph, we denote by $|p|$ its length. We say that a path is \emph{trivial} if it consists of a single point.
\begin{lem}\label{l1}
There exists an infinite cyclic subgroup $H\h G$ such that
\begin{equation}\label{int}
FHF^{-1}\cap H=\{1\}.
\end{equation}
In particular, we have $H \cap F=\emptyset$.
\end{lem}

\begin{proof}
Let $n=|F|+1$. By Theorem \ref{Eg}, there exists a free subgroup $B\h G$ of rank $n^3$. Let $B=B_1\ast  \cdots \ast B_{n}$, where $B_s$ is free of rank $n^2$ for each $s=1, \ldots , n$. Combining (\ref{1F}) with the observation that $B_s\cap B_t=\{1\}$ whenever $s\ne t$, we conclude that there is $m\in \{1, \ldots, n^2\}$ such that
\begin{equation}\label{Bm}
F\cap B_m=\emptyset.
\end{equation}
It is easy to see that $B_m \h B$ (see Example \ref{bex} (b)). By \cite[Proposition 4.35]{DGO}, being a hyperbolically embedded subgroup is a transitive relation. Therefore, $B_m \h G$.

Let $b_1, \ldots , b_{n^2}$ be a basis of $B_m$. We are going to show that (\ref{int}) holds for some $H=\langle b_i\rangle$. Arguing by contradiction, assume that for every $i=1, \ldots, n^2$, there is $(f_i, g_i)\in F\times F^{-1}$, $k_i\in \mathbb Z$, and $\ell_i\in \mathbb Z\setminus \{0\}$ such that $f_ib_i^{k_i}g_i=  b_i^{\ell_i}$. Since $n^2>|F|^2$, we have $(f_i,g_i)=(f_j,g_j)$ for some $i\ne j$. It follows that
$$
f_ib_i^{k_i}b_j^{-k_j}f_i^{-1} = f_ib_i^{k_i}g_i \cdot (f_jb_j^{k_j}g_j)^{-1} =b_i^{\ell_i}b_j^{-\ell_j}\in B_m\setminus \{1\}.
$$
Applying Proposition \ref{malnorm} to the hyperbolically embedded subgroup $B_m$ we obtain $f_i\in B_m$, which contradicts (\ref{Bm}). Thus, there exists $i$ such that (\ref{int}) holds for $H=\langle b_i\rangle$. As above, we have $H\h B_m\h G$ and, therefore, $H\h G$ by transitivity.
\end{proof}

From now on, we fix a subgroup $H\le G$ satisfying the conclusion of Lemma \ref{l1}. Let $X$ be a generating set of $G$ such that $H\h (G,X)$. The property $H\h (G,X)$ is not sensitive to adding a finite set of elements to $X$ (see \cite[Corollary 4.27]{DGO}) and hence we can assume that
\begin{equation}\label{FX}
F\subseteq X
\end{equation}
without loss of generality. Let $$\mathcal A=X\sqcup H$$ be the associated generating alphabet of $G$. We denote by $\d_{\mathcal A}$ and $|\cdot |_{\mathcal A}$ the corresponding word metric and word length on $G$.

Let $C$ denote the constant provided by Proposition \ref{n-gon}. We fix any $t\in H$ satisfying
\begin{equation}\label{dlt}
\dl (1,t)> 5C;
\end{equation}
such an element exists by condition (b) of Definition \ref{hedefn}.

The next two lemmas are proved under the following common notation and assumptions. Let $w$ be a word in the alphabet $\mathcal A$ of the form $$w=tf_1tf_2\ldots tf_n,$$ where $f_1, \ldots, f_n$ are some letters from $F$ (note that we use the assumption (\ref{FX}) here). Let $p$ denote any path in $\G$ with label $w$. Then $p$ decomposes as
\begin{equation}\label{cd}
p=a_1b_1\ldots a_nb_n,
\end{equation}
where $a_i$, $b_i$ are edges with labels
\begin{equation}\label{aibi}
\lab (a_i)=t,\;\; \lab(b_i)=f_i
\end{equation}
for all $i=1, \ldots, n$. We call (\ref{cd}) the \emph{canonical decomposition} of $p$.

\begin{lem}\label{l2}
For each $i\in \{ 1, \ldots, n\}$, $a_i$ is an isolated $H$-component of $p$.
\end{lem}

\begin{figure}
  \begin{center}
  \scalebox{1.25}{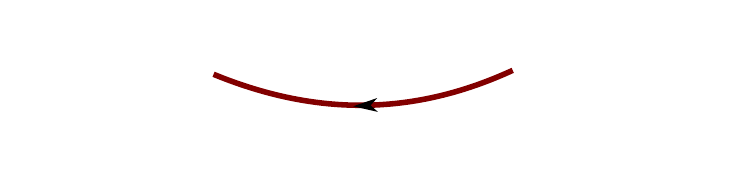}
\caption{Proof of Lemma \ref{l2}.}\label{fig0}
  \end{center}
\end{figure}

\begin{proof}
The claim that every $a_i$ is an $H$-component of $p$ immediately follows from (\ref{int}). Thus we only need to show that all these $H$-components are isolated. Arguing by contradiction, consider a pair of indices $i<j$ with minimal possible value of $k=j-i$ such that $a_i$ is connected to $a_j$. If $k=1$, then we have $f_i\in H$, which contradicts the choice of $H$ (see Lemma \ref{l1}). Thus $k\ge 2$. Let $e$ be the edge labelled by an element of $H$ and connecting $(a_j)_-$ to $(a_i)_+$ in $\G$ (see Fig. \ref{fig0}). Let also $r$ denote the segment of $p$ bounded by $(a_i)_+$ and $(a_j)_-$.

By minimality of $k$, all $H$-components $a_s$ for $i<s<j$ are isolated in the loop $re$. We think of $re$ as a geodesic $2k$-gon with sides $b_i, a_{i+1}, b_{i+1}, \ldots,  a_{j-1}, b_{j-1}, e$. Note that $\dl ((a_s)_-, (a_s)_+)=\dl (1,t)$ for all $s$ (see (\ref{defdl})). Applying Proposition \ref{n-gon} we obtain
$$
\dl (1,t)=\frac{1}{k-1} \sum_{s=i+1}^{j-1} \dl ((a_s)_-, (a_s)_+) \le C \frac{2k}{k-1}\le 4C,
$$
which contradicts (\ref{dlt}).
\end{proof}

\begin{lem}\label{l3}
Let $p$ be as above and let $q$ be a geodesic in $\G$ connecting $p_-$ to $p_+$. Then $q=c_1d_1\ldots c_nd_n$, where $d_i\ne 1$ and $c_i$ is an $H$-component of $q$ connected to $a_i$ for all $i\in \{ 1, \ldots, n\}$. In particular, the path $p$ is geodesic.
\end{lem}

\begin{proof}
We prove the lemma by induction on $n$. The loop $pq^{-1}$ can be thought of as a geodesic quadrilateral with sides $q^{-1}$, $a_1$, $b_1$, and $a_2b_2\ldots b_na_n$ (the last side is geodesic by the inductive assumption if $n>1$ and reduces to a point if $n=1$). By Proposition \ref{n-gon} and (\ref{dlt}), $a_1$ cannot be an isolated $H$-component in $pq^{-1}$. By Lemma \ref{l2}, $a_1$ cannot be connected to another $H$-component of $p$; therefore, it is connected to an $H$-component $c_1$ of $q$. As $q$ is geodesic, $c_1$ must be the first edge of $q$.

If $n=1$, we obtain $q=c_1d_1$. Clearly, where $d_1\ne 1$ as otherwise we would have $t\in H$, which contradicts the choice of $H$ (see Lemma \ref{l1}). If $n>1$, we continue as follows. Assume that for some $1\le i<n$, we have $q=c_1d_1\ldots c_iq^\prime$, where $c_1, \ldots, c_i$ are $H$-components of $q$ connected to $a_1, \ldots a_i$, respectively. Let $f_i$ be the edge labeled by an element of $H$ (or the trivial path) and connecting $(c_i)_+$ to $(a_i)_+$. Applying Proposition \ref{n-gon} to the geodesic pentagon with sides $b_{i}$, $a_{i+1}$, $b_{i+1}\ldots a_nb_n$, $(q^\prime)^{-1}$, $f_i$, we conclude that $a_{i+1}$ cannot be isolated in it. By Lemma \ref{l2}, $a_{i+1}$ is connected to an $H$-component $c_{i+1}$ of $q^\prime$. By induction, we obtain that $q=c_1d_1\ldots c_nd_n$, where $c_i$ is connected to $a_i$ for all $i\in \{ 1, \ldots, n\}$. It remains to note that if $d_n$ is trivial then we have $t\in H$, which contradicts Lemma \ref{l1} as above, and if $d_i$ reduces to a point or is labeled by an element of $H$ for some $i\in \{1, \ldots, n-1\}$ then $a_i$ and $a_{i+1}$ are connected, which contradicts Lemma \ref{l2}. In particular, each $c_i$ is a component of $q$ and $|q|\ge |p|$.
\end{proof}

\begin{figure}
  \begin{center}
  \scalebox{1.25}{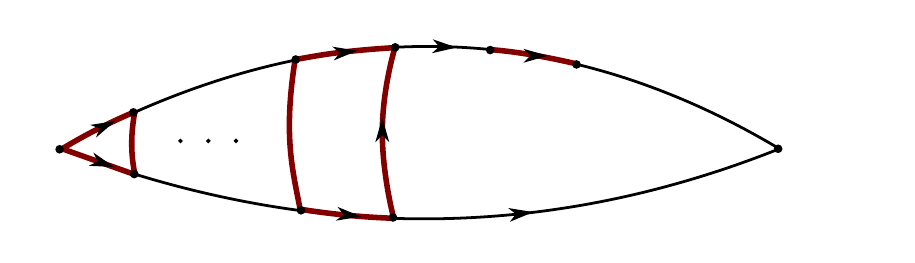}
\caption{Proof of Lemma \ref{l3}.}\label{fig1}
  \end{center}
\end{figure}

\begin{lem}\label{l4}
The set $tF$ freely generates a free subsemigroup of $G$.
\end{lem}

\begin{proof}
It suffices to show that if some words $w=tf_1tf_2\ldots tf_n$ and $u =tg_1tg_2\ldots tg_m$ in the alphabet $\mathcal A$, where $f_1, \ldots, f_n, g_1, \ldots, g_m\in F$, represent the same element in $G$ then the words $w$ and $u$ are equal, i.e., $m=n$ and we have
\begin{equation}\label{u=v}
f_1=g_1,\; \ldots,\; f_n=g_n
\end{equation}
in $H$.

Let $p$ and $q$ be paths in $\G$ starting at $1$ and labelled by $w$ and $u$, respectively, and let
$$
p=a_1b_1\ldots a_nb_n,\;\;\;{\rm and}\;\;\;  q=c_1d_1\ldots c_md_m
$$
be their canonical decompositions. By Lemma \ref{l3} we have $m=n$ and $a_i$ is connected to $c_i$ for all $i\in \{1, \ldots, n\}$. Let $e_i$, $f_i$ be edges of $\G$ labelled by elements of $H$ and connecting $(a_i)_- $ to $(c_i)_-$ and $(c_i)_+$ to $(a_i)_+$, respectively (see Fig. \ref{fig2}). Reading the label of the cycle $f_i b_ie_{i+1}d_i^{-1}$ and using (\ref{int}), we obtain $$\lab (f_i) \in FHF^{-1}\cap H=\{1\}$$ for all $1\le i \le n$ (for $i=n$, we read the label of the triangle $f_nb_nd_n^{-1}$). Thus, we have $(a_i)_+=(c_i)_+$ for all $1\le i \le n$. Similarly, $(a_i)_-=(c_i)_-$ for all $1\le i \le n$ . This obviously implies (\ref{u=v}).
\end{proof}

\begin{figure}
  \begin{center}
  \scalebox{1.25}{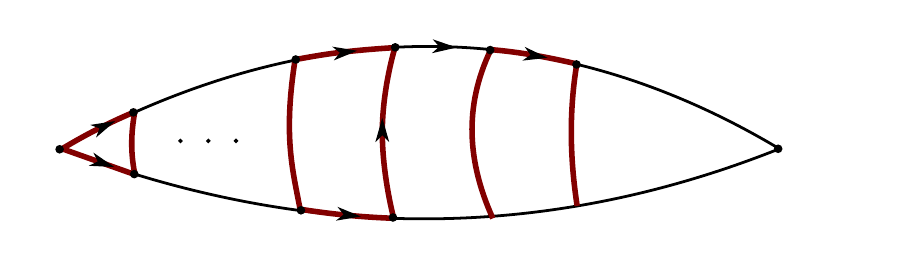}
\caption{Proof of Lemma \ref{l4}.}\label{fig2}
  \end{center}
\end{figure}

We construct the required generalized combing $C\colon G\times G\to \mathcal P(G)$ as follows. Let
$$
\Omega = \{ t^{\pm 1}\} \cup F^{\pm 1} \cup \{ h\in H \mid \dl (1, h)\le 5C\}
$$
The set $\Omega $ is finite as $H\h G$ (see condition (b) in Definition \ref{hedefn}). For a subgraph $\Delta $ of $\G$, we denote by $V(\Delta )$ the set of its vertices considered as a subset of $G$.

Recall that $S$ denotes the free subsemigroup of $G$ generated by $tF$. For each $g\in G$, we fix a geodesic path $\gamma _g$ in $\G$ connecting $g$ to $1$ such that
\begin{enumerate}
\item[($\ast$)] for every $s\in S\setminus\{1\}$, $\lab(\gamma _s)=tf_1\ldots tf_n$, where $f_1, \ldots, f_n\in F$.
\end{enumerate}
Note that we can always ensure ($\ast$) by Lemma \ref{l3}. Further, we define
\begin{equation}\label{Cs}
C(1,g) =  V(\gamma_g \cup g\gamma_{g^{-1}})\cdot \Omega\cdot \Omega.
\end{equation}
Informally, $C(1,g)$ is the $2$-neighborhood (with respect to the word metric associated to $\Omega$) of the set of vertices of the loop $\gamma_g \cup g\gamma_{g^{-1}}$. Finally, we let
\begin{equation}\label{Ce}
C(f,g)= fC(1, f^{-1}g)
\end{equation}
for all $f, g\in G$.

We also define the length function $\ell\colon G\to [0,+\infty)$ by the equation
$$
\ell(G)=|g|_{\mathcal A}.
$$

The next three lemmas finish the proof of Proposition \ref{combing}.

\begin{lem}
The generalized combing $C\colon G\times G\to \mathcal P(G)$ is symmetric and $G$-equivariant.
\end{lem}

\begin{proof}
$G$-equivariance of $C$ immediately follows from (\ref{Ce}). To prove that $C$ is symmetric, we first note that
$$
C(1,g)=V(\gamma_g \cup g\gamma_{g^{-1}})\Omega^2=g V(g^{-1}\gamma_g \cup \gamma_{g^{-1}})\Omega^2 =gC(1, g^{-1})=C(g,1)
$$
by (\ref{Cs}) and equivariance, and then use (\ref{Ce}).
\end{proof}

\begin{lem}
Let $M=\max\limits_{u\in \Omega^2} |u|_\mathcal A$. Then the growth functions associated to $C$ and $\ell$ satisfy
\begin{equation}\label{gsubl}
\gamma (n)\le 2(n+M+1)|\Omega |^2
\end{equation}
and
\begin{equation}\label{dsubl}
\rho (n) \le n + M.
\end{equation}
\end{lem}

\begin{proof}
Let $g\in G$. If $a\in C(1,g)$ and $|a|_\mathcal A\le n$, then by definition of $C(1,g)$ there exists $b\in  V(\gamma_g \cup g\gamma_{g^{-1}})$ and $w\in \Omega^2$ such that $a=bw$. In particular,
\begin{equation}\label{b}
|b|_\mathcal A\le |a|_\mathcal A + |w|_\mathcal A \le n+M.
\end{equation}
Since  $B=\gamma_g \cup g\gamma_{g^{-1}}$ is a union of two geodesics connecting $1$ and $g$, there are at most $2(n+M+1)$ vertices $b\in V(B)$ satisfying (\ref{b}). Therefore, there are at most $2(n+M+1)|\Omega |^2$ possibilities for $a$, which gives (\ref{gsubl}).

The inequality (\ref{dsubl}) follows immediately from the obvious fact that for every $g\in G$ and every $c\in C(1,g)$, we have $\d_\mathcal A (1,c)\le |g|_\mathcal A +M$.
\end{proof}

\begin{lem}\label{lemC}
For every $g\in G$ and $s\in S$, we have $C(1,s) \cap C(s,g) \cap C(1,g) \neq \emptyset$.
\end{lem}

\begin{figure}
  \begin{center}
 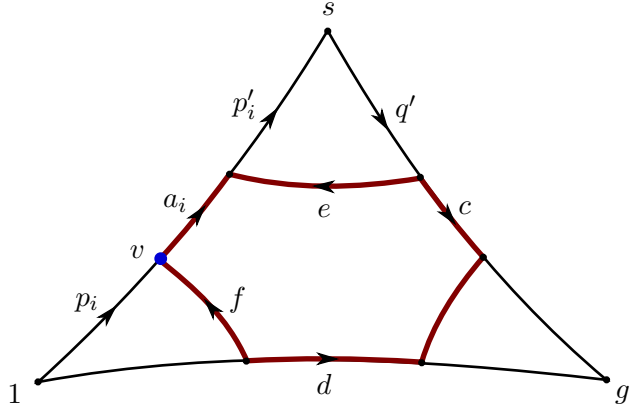
\caption{Case 1 in the proof of Lemma \ref{lemC}.}\label{fig3}
  \end{center}
\end{figure}

\begin{proof}
Let $s$ (respectively, $g$) be any element of $S$ (respectively, $G$). If $s=1$, then we have $1\in C(1,s) \cap C(s,g) \cap C(1,g)$. Thus, we can assume that $s\ne 1$ without loss of generality.

Consider the geodesic triangle $\Delta $ with vertices $1$, $s$, $g$, and sides $p=\gamma_s$, $q=s\gamma_{s^{-1}g}$, and $r=\gamma_g$. To prove the lemma it suffices to find a vertex $v\in V(p)$ such that
\begin{equation}\label{vV}
v\in V(q)\Omega^2 \cap  V(r)\Omega^2.
\end{equation}

Since $p=\gamma_s$ and $s\in S$, the path $p$ decomposes as in (\ref{cd}), (\ref{aibi}), where $f_1, \ldots, f_n\in F$. For each $1\le i\le n$, we denote by $p_i$ and $p_i^\prime$ the (possibly trivial) subpaths of $p$ such that $p=p_ia_ip_i^\prime$ and think of $p_ia_ip_i^\prime qr^{-1}$ as a geodesic pentagon. Proposition \ref{n-gon} and inequality (\ref{dlt}) imply that $a_i$ cannot be a isolated $H$-component in $\Delta$.

The following elementary observation will be used several times in this proof without explicit references: no two distinct $H$-components of a geodesic path in $\G$ can be connected (indeed, otherwise these two $H$-components and the segment of the geodesic between them could be replaced with a single edge, which contradicts geodesicity).  In particular, $a_i$ cannot be connected to another $H$-component of $p$ and therefore it is connected to an $H$-component of $q$ or $r$. We consider several cases.

{\it Case 1.} There is $1\le i\le n$ such that $a_i$ is connected to an $H$-component $c$ of $q$ and an $H$-component $d$ of $r$.

Let $e$ (respectively $f$) be the edge in $\G$ (or the trivial path) labelled by an element of $H$ and connecting $c_-$ to $(a_i)_+$ (respectively, $d_-$ to $(a_i)_-$). We denote by $q^\prime$ the subpaths of $q$ connecting $s$ to $c_-$ (see Fig. \ref{fig3}). Note that $e$ is isolated in the geodesic triangle $p_i^\prime q^\prime e$. By Proposition \ref{n-gon}, we have $\dl (1, \lab(e))\le 3C$; in particular, $\lab(e)\in \Omega$. Similarly, $\lab (f)\in \Omega$. Thus the vertex $v=(a_i)_-$ satisfies (\ref{vV}).
\begin{figure}
  \begin{center}
 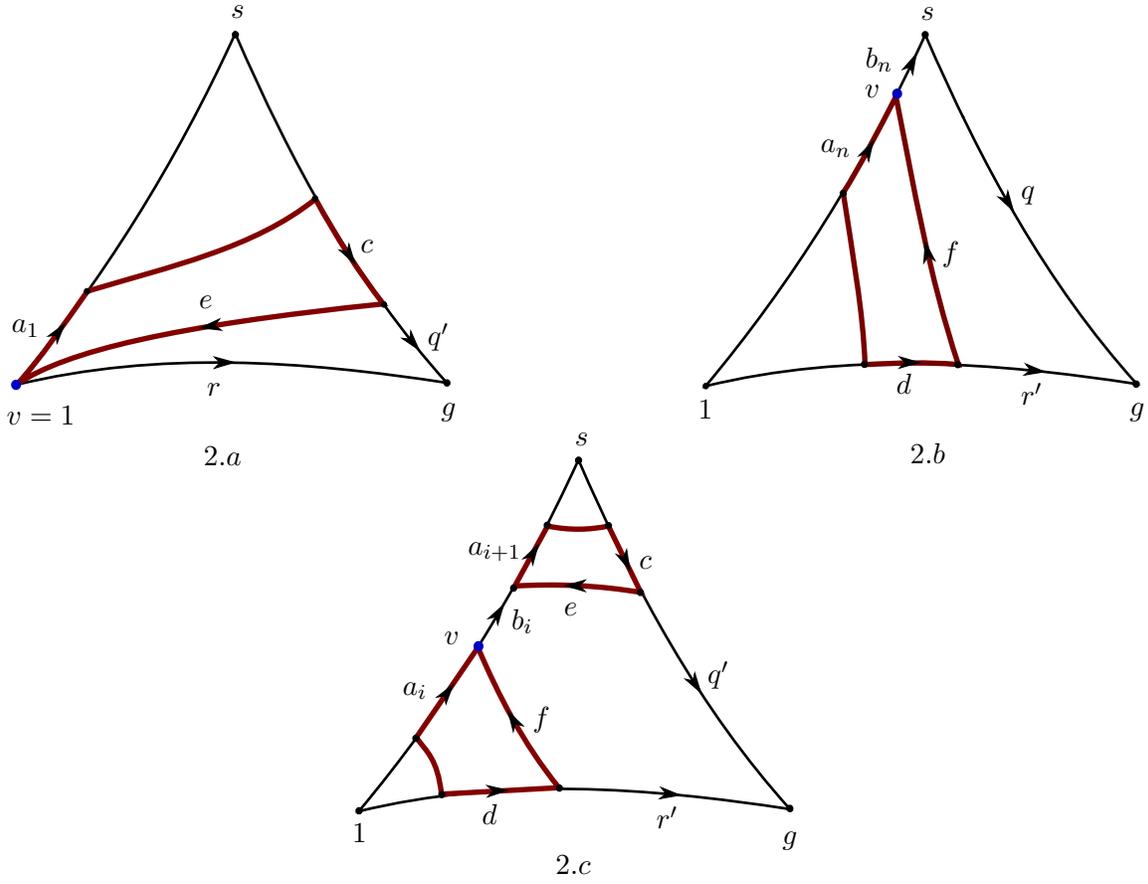
\caption{Cases 2.a-2.c in the proof of Lemma \ref{lemC}.}\label{fig4}
  \end{center}
\end{figure}

{\it Case 2.} Suppose that no $a_i$ is connected to $H$-components of both $q$ and $r$. In turn, this case subdivides into three subcases (see Fig. \ref{fig4}).

{\it Case 2.a.} First assume that $a_1$ is connected to an $H$-component $c$ of $q$. Let $e$ be the edge in $\G$ (or the trivial path) labelled by an element of $H$ and connecting $c_+$ to $(a_1)_-$. Let $q^\prime$ be the segment of $q$ bounded by $c_+$ and $g$. By our assumption, $a_1$ cannot be connected to an $H$-component of $r$; therefore, $e$ is isolated in the geodesic triangle $er(q^\prime)^{-1}$. As in Case 1, Proposition \ref{n-gon} implies $\lab(e)\in \Omega$. It follows that the vertex $v=(a_1)_-=1$ satisfies (\ref{vV}).

{\it Case 2.b.} Now suppose that $a_n$ is connected to an $H$-component $d$ of $r$. Let $f$ be the edge in $\G$ (or the trivial path) labelled by an element of $H$ and connecting $d_+$ to $(a_n)_+$. As above, we obtain $\lab(f)\in \Omega$. Note also that $f_n=\lab (b_n)\in \Omega$ by the definition of $\Omega$. It follows that the vertex $v=(a_n)_+$ satisfies (\ref{vV}).

{\it Case 2.c.} Finally assume that $a_1$ is connected to an $H$-component of $r$ and $a_n$ is connected to an $H$-component of $q$. Then there exists $1\le i<n$ such that $a_i$ is connected to an $H$-component $d$ of $r$ and $a_{i+1}$ is connected to an $H$-component $c$ of $q$.

Let $e$ (respectively, $f$) be the edge in $\G$ (or the trivial path) labelled by an element of $H$ and connecting $c_+$ to$(a_{i+1})_-$ (respectively, $d_+$ to $(a_i)_+$). Let $q^{\prime}$ and $r^{\prime}$ be the segments of $q$ and $r$ going from $c_+$ to $g$ and from $d_+$ to $g$, respectively. Note that $e$ and $f^{-1}$ are isolated in the geodesic pentagon $eb_i^{-1}f^{-1}r^{\prime}(q^{\prime})^{-1}$. By Proposition \ref{n-gon}, we have $\dl (1, \lab(e))\le 5C$ and hence $\lab(e)\in \Omega$. Similarly, $\lab(f)\in \Omega$. Taking $v$ to be either of the endpoints of $b_i$, we obtain (\ref{vV}).
\end{proof}

\section{Proof of the main theorem}

In this section, we introduce an auxiliary class of groups, denoted by $\mathcal C$, which includes all acylindrically hyperbolic groups with trivial finite radical and is closed under taking direct products. We then show that  $\sr(\C (G))=1$ for any $G\in \mathcal C$.

\begin{defn}\label{defC}
Let $\mathcal C$ be the class of all groups $G$ with the following property. For any finite subset $F\subset G$, there exists a pseudolength function $\ell$ on $G$, an element $t\in G$, and a symmetric $G$-equivariant generalized combing $C\colon G\times G\to \mathcal P(G)$ such that
\begin{enumerate}
\item[(a)] the set $tF$ freely generates a free subsemigroup $S$ of $G$;
\item[(b)] $C(1,s) \cap C(s,g) \cap C(1,g) \neq \emptyset$ for all $g\in G$ and $s\in S$;
\item[(c)] the growth functions $\gamma (n)$ and $\rho(n)$ of $C$ computed with respect to $\ell$ are bounded by some polynomials in $n$ from above.
\end{enumerate}
\end{defn}

\begin{lem}\label{prod}
The class of groups $\mathcal C$ is closed under taking finite direct products.
\end{lem}

\begin{proof}
It suffices to prove that $G=G_1\times G_2\in \mathcal C$ for any $G_1, G_2\in \mathcal C$.

Given any finite $F\subset G$, we have $F\subset F_1\times F_2$, where $F_1$ and $F_2$ are the projections of $F$ to $G_1$ and $G_2$, respectively. Since $G_1, G_2\in \mathcal C$, there exist elements $t_1\in G_1$ and $t_2\in G_2$ such that $t_1F_1$ and $t_2F_2$ freely generate free subsemigroups. Let $t_1F_1=\{f_i\}_{i}$, $t_2F_2=\{g_j\}_{j}$. Assume that
$$(f_{i_1},g_{j_1})\dots (f_{i_n},g_{j_n})=(f_{k_1}, g_{l_1}) \dots (f_{k_m},g_{l_m})$$
or, equivalently,
$$f_{i_1} \dots f_{i_n}=f_{k_1} \dots f_{k_m} \text{ and } \; \; g_{j_1} \dots g_{j_n}=g_{l_1} \dots g_{l_m}.$$
Since $t_1F_1$ and $t_2F_2$ freely generate free subsemigroups, we have $n=m$ and $i_s=k_s$, $j_s=l_s$ for all $s\in \{1, \dots, n\}.$ Thus, the set $(t_1,t_2)(F_1\times F_2)$ freely generates a free subsemigroup of $G$ and, therefore, so does $(t_1,t_2)F\subseteq tF_1\times tF_2$. This proves condition (a) from Definition \ref{defC} for $t=(t_1,t_2)$.

Further, let $\ell_1$, $\ell_2$ (respectively, $C_1$, $C_2$) be pseudolength functions (respectively, symmetric equivariant generalized combings) on $G_1$ and $G_2$ such that, for $i=1,2$, we have
\begin{equation}\label{Ci}
C_i(1,s) \cap C_i(s,g) \cap C_i(1,g) \neq \emptyset
\end{equation}
for all $g\in G_i$ and $s\in S_i$, where $S_i$ is the subsemigroup of $G_i$ generated by $t_iF_i$, and the corresponding growth functions $\gamma_i$ and $\rho _i$ are bounded by some polynomials from above.
We define a generalized combing $C\colon G\times G\to \mathcal P(G)$ by the rule
$$
C((x_1, x_2),(y_1,y_2))= C_1(x_1,y_1)\times C_2(x_2, y_2)
$$
for all $(x_1,x_2),(y_1,y_2)\in G=G_1\times G_2$. It is straightforward to verify that  $C$ is $G$-equivariant, symmetric, and (\ref{Ci}) implies part (b) of Definition \ref{defC}.

Let also
$$
\ell(x,y)= \max\{\ell(x), \ell(y)\}
$$
for all $(x,y)\in G$. Clearly, $\ell$ is a pseudolength function on $G$. Let $B_i(n)$, $i=1,2$, and $B(n)$ denote the balls of radius $n$ centered at $1$ in $G_i$ and $G$, respectively, with respect to the length functions $\ell_i$ and $\ell$. Then we have $B(n)=B_1(n)\times B_2(n)$ for all $n\in \mathbb N$. Using the definition of $C$, for every $g=(x,y)\in G$ we obtain
\begin{align*}
B(n)\cap C(1,g) = & \left(B_1(n)\times B_2(n)\right)\cap\left(C_1(1,x)\times C_2(1,y)\right) = \\
& \left(B_1(n)\cap C_1(1,x)\right)\times \left(B_2(n)\cap C_2(1,y)\right).
\end{align*}
This implies
\begin{equation}\label{gn}
\gamma (n) \le \gamma _1(n)\gamma_2(n)
\end{equation}
for all $n\in \mathbb N$.
Similarly, for every $g=(x,y)\in B(n)$ and every $v=(u,w)\in C(1,g)$, we have
$$
\ell(v)=\max\{\ell_1(u), \ell_2(w)\}\le \max\{\rho_1(n), \rho_2(n)\},
$$
which implies
\begin{equation}\label{rn}
\rho(n)\le \rho_1(n)\rho_2(n)
\end{equation}
for all $n\in \mathbb N$.
Clearly, inequalities (\ref{gn}) and (\ref{rn}) together with our assumptions about $\gamma_i$ and $\rho _i$ imply condition (c) from Definition \ref{defC}.
\end{proof}

Our proof of Theorem \ref{main} is based on a sufficient condition for a group $G$ to satisfy $\sr(\C(G))=1$ obtained in \cite{DH} (see also \cite{Ror97}).

\begin{defn}
Let $G$ be a group and let $a\in \mathbb CG$. Following \cite{DH}, we say that $a$ has the \emph{$\ell^2$-spectral radius property} if $r_2(a)=r(a)$.
\end{defn}

\begin{thm}[{\cite[Theorem 1.4]{DH}}]\label{sr1}
Suppose that for any finite subset $F$ of a group $G$, there exists $t\in G$ such that $tF$ generates a free subsemigroup and every $a\in \mathbb{C} G$ with $supp(a)\subset tF$ has the $\ell^2$-spectral radius property. Then $\sr(\C(G))=1.$
\end{thm}

Combining this theorem with Proposition \ref{CR}, we obtain the following.

\begin{cor}\label{classC}
For any group $G\in \mathcal C$, we have $\sr(\C(G))=1.$
\end{cor}

\begin{proof}
Let $G\in \mathcal C$ and let $F\subseteq  G$ be a finite subset. Let $t\in G$ and $C\colon G\times G\to \mathcal P(G)$ be as in Definition \ref{defC}. By Proposition \ref{CR}, every $a\in \mathbb{C} G$ with $supp(a)\subset tF$ has the $\ell^2$-spectral radius property. Now Theorem \ref{sr1} gives us the desired result.
\end{proof}

We are finally ready to prove our main result.

\begin{proof}[Proof of Theorem~\ref{main}]
By Proposition \ref{combing}, any acylindrically hyperbolic group $G$ with trivial finite radical belongs to $\mathcal C$. Hence finite direct products of such groups are in $\mathcal C$ by Lemma \ref{prod}. It remains to apply Corollary \ref{classC}.
\end{proof}

\vspace{1cm}

\noindent \textbf{Maria Gerasimova:} Department of Mathematics, Bar-Ilan University, Ramat Gan, Israel.\\
E-mail: \emph{mari9gerasimova@mail.ru}

\vspace{.5cm}

\noindent \textbf{Denis Osin:} Department of Mathematics, Vanderbilt University, Nashville, U.S.A.\\
E-mail: \emph{denis.v.osin@vanderbilt.edu}

\end{document}